\documentclass[11pt,oneside,reqno]{amsart}
\usepackage{graphicx}

\newcommand{\bea}{\begin{eqnarray}}
\newcommand{\eea}{\end{eqnarray}}
\newcommand{\ba}{\begin{array}}
\newcommand{\ea}{\end{array}}
\newcommand{\edc}{\end{document}}
\newcommand{\bc}{\begin{center}}
\newcommand{\ec}{\end{center}}
\newcommand{\be}{\begin{equation}}
\newcommand{\ee}{\end{equation}}

\def\bc{{\mathbb C}}

\newtheorem{thm}{Theorem}[section]

\newtheorem{prop}[thm]{Proposition}

\theoremstyle{remark}
\newtheorem{rem}{Remark}[section]

\date{\today}
\begin{document}
\title[Gibbs Measures with memory of length 2]{Phase transition and Gibbs Measures of Vannimenus model on semi-infinite Cayley tree of order three}

\author{Hasan Ak\i n}
\address{Department of Mathematics, Faculty of Education,
Zirve University, 27260 Gaziantep, Turkey}

\date{May 19, 2016}%
\begin{abstract}
Ising model with competing nearest-neighbors and prolonged
next-nearest-neighbors interactions on a Cayley tree has long been
studied but there are still many problems untouched. This paper
tackles new Gibbs measures of Ising-Vannimenus model with
competing nearest-neighbors and prolonged next-nearest-neighbors
interactions on a Cayley tree (or Bethe lattice) of order three.
By using a new approach, we describe the translation-invariant
Gibbs measures for the model. We show that some of the measures
are extreme Gibbs distributions. In this paper we take up with
trying to determine when phase transition does occur.
\end{abstract}
\maketitle

\section{Introduction}



The definition of a Gibbs state on a finite subset of
$\mathbb{Z}^d$ goes back to the classical work of Gibbs
\cite{Gibbs}. Markov random fields on the euclidean lattices
$\mathbb{Z}^d$ were first introduced by Dobrushin \cite{D1}.
Preston has shown that Markov random fields and Gibbs states with
nearest neighbour potentials are the same \cite{Preston}. Gibbs
states (or measures) only consider finite subsets of
$\mathbb{Z}^d$ which are then used to compute various
thermodynamic quantities and examine their corresponding limiting
behaviour \cite{D1,G,Bax,FV}. Fannes and Verbeure \cite{FV} took
into account correlations between $n$ successive lattice points as
they studied one-dimensional classical lattice systems with an
increasing sequence of subsets of the state space. These states
correspond in probability theory to so-called Markov chains with
memory of length $n$.

Two important advantages of using tree models to determine Gibbs
measures are that they eliminate the need for approximations and
calculations can be carried out to high degrees of accuracy. In
addition, models such as Ising and Potts on the Cayley tree (or
Bethe lattice) can be helpful in discovering additional systems
with related properties. As a result, many researchers have
employed the Ising and Potts models in conjunction with the Cayley
tree \cite{BRZ,PB,bak,BB,BG,GTA,GATTMJ,UGAT}. The Ising model has
relevance to physical, chemical, and biological systems
\cite{FS,G,IT,Iof}. The Ising model investigated by Vannimenus
\cite{Vannimenus} consists of Ising spins $(\sigma=\pm 1)$ on a
rooted Cayley tree with a branching ratio of 2 \cite{ART}, in
which two coupling constants are present: nearest-neighbour (NN)
interactions of strength and next-nearest-neighbour (NNN)
interactions. Specifically, in \cite{Akin2016}, the author has
used a new method to investigate a rigorous description of Gibbs
measures with a memory of length 2 that corresponds to the
Ising-Vannimenus Model on the Cayley tree of order 2. This present
paper introduces the Ising model corresponding to the Hamiltonian
given by Vannimenus \cite{Vannimenus} on Cayley tree of order
three. Furthermore, the author \cite{Akin2016} has proposed a
rigorous measure-theoretical approach to investigate Gibbs
measures with a memory of length for the Ising-Vannimenus Model on
the Cayley tree of order two. This study further bases its
investigation of Gibbs measures on the Markov random field on
trees and on recurrent equations following from this theory
\cite{BRZ,BG,GATTMJ,ART,AT1,NHSS,GRS,Zachary,NHSS1}. Rozikov et
al. \cite{RAU} analyzed the recurrent equations of a generalized
Axial Next-Nearest-Neighbour Ising (ANNNI) model on a Cayley tree
and documented critical temperatures and curves, number of phases,
and partition function. To describe all Gibbs measures
corresponding to a given Hamiltonian is one of the main problems
of statistical physics \cite{NHSS1}.

In this paper, we are going to focus on the translation invariant
Gibbs measures with memory of length 2 associated to the
Ising-Vannimenus model on a Cayley tree of order 3. One of many
approaches to studying the equation solutions that describe Gibbs
measures for lattice models on Cayley tree is the Markov random
field \cite{BG,ART,AT1,NHSS}. This paper uses the Markov random
field to achieve the following objectives: construct the
recurrence equations corresponding to a generalized ANNNI model;
formulate the problem in terms of nonlinear recursion relations
along the branches of a Cayley tree of order three; fulfill the
Kolmogorov \emph{consistency} condition; describe the
translation-invariant Gibbs measures for the model; and show that
some measures are extreme Gibbs distributions.

In \cite{NHSS} the authors have studied the problem of phase
transition for models considered by Vannimenus \cite{Vannimenus}.
Mukhamedov et al. \cite{MDA} have proved the existence of the
phase transition for the Vannimenus model \cite{Vannimenus} in the
$p$-adic setting. Ganikhodjaev \cite{Nasir} has considered the
Ising model on the semi-infinite Cayley tree of second order with
competing interactions up to the third-nearest-neighbors with
spins belonging to the different branches of the tree and for this
model investigated the problem of phase transition. Significant
research has determined that a finite graph corresponds to exactly
one Gibbs state with potential $F$ for a given potential $F$ and
that graphs that are not finite lack this quality, \emph{i.e.},
for some potentials $F$, there may be more than one corresponding
Gibbs state with potential $F$ \cite{G,Preston,Sinai}. When there
is more than one corresponding Gibbs measure, we say that phase
transition occurs for the potential $F$. In this paper, we also
attempt to determine when phase transition occurs for the model.

This article is organized as follows: In Section
\ref{Preliminaries}, we provide definitions and preliminaries. In
Section \ref{Gibbs measures}, we introduce general structure of
Gibbs measures with memory of length 2 on a Cayley tree of order
3, with functional equations, and fulfill the Kolmogorov
\emph{consistency} condition. In Section \ref{TIGM}, we establish
translation-invariant Gibbs measures corresponding to the
associated model \eqref{hm}, demonstrating that some occurrences
are extreme. Finally, Section \ref{Conclusions} contains
concluding remarks and discussion of the consequences of the
results.

\section{Preliminaries and Definitions}\label{Preliminaries}
For this paper, let  $\Gamma^k=(V, L, i)$ be the uniform Cayley
tree of order $k$ with a root vertex $x^{(0)}\in V$, where each
vertex has $k + 1$ neighbors with $V$ as the set of vertices and
the set of edges. The notation $i$ represents the incidence
function corresponding to each edge $\ell\in L$, with end points
$x_1,x_2\in V$. There is a distance $d(x,y)$ on $V$ the length of
the minimal point from $x$ to $y$, with the assumed length of 1
for any edge.

We denote the sphere of radius $n$ on $V$ by
$$
W_n=\{x\in V: d(x,x^{(0)})=n \}
$$
and the ball of radius $n$ by
$$
V_n=\{x\in V: d(x,x^{(0)})\leq n \}.
$$
The set of direct successors of $x$ for any $x\in W_n$ is denoted
by
$$S(x)=\{y\in W_{n+1}: d(x,y)=1 \}.
$$
The Ising  model with competing nearest-neighbors interactions is
defined by the Hamiltonian
\begin{equation*}\label{hm}
H(\sigma)=-J\sum_{<x,y>\subset V}\sigma(x)\sigma(y),
\end{equation*}
where the sum runs over nearest-neighbor vertices $<x,y>$ and the
spins $\sigma(x)$ and $\sigma(y)$ take values in the set
$\Phi=\{-1,+1\}$.

A finite-dimensional distribution of measure $\mu$  in the volume
$V_n$ has been defined by formula
\begin{equation*}\label{mu}
\mu_n(\sigma_n)=\frac{1}{Z_{n}}\exp[-\frac{1}{T}H_n(\sigma)+\sum_{x\in
W_{n}}\sigma(x)h_{x}]
\end{equation*}
with the associated partition function defined as
\begin{equation*}\label{mu}
Z_n=\sum_{\sigma_n \in
\Phi^{V_n}}\exp[-\frac{1}{T}H_n(\sigma)+\sum_{x\in
W_{n}}\sigma(x)h_{x}],
\end{equation*}
where the spin configurations $\sigma_n$ belongs to $\Phi^{V_n}$
and $h=\{h_x\in \mathbb{R},x\in V\}$ is a collection of real
numbers that define boundary condition (see \cite{BG,GRRR,GHRR}).
Previously, researchers frequently used memory of length 1 over a
Cayley tree to study Gibbs measures \cite{BG,GRRR,GHRR}.

The Hamiltonian
\begin{equation}\label{hm}
H(\sigma)=-J_p\sum_{>x,y<}\sigma(x)\sigma(y)-J\sum_{<x,y>}\sigma(x)\sigma(y)
\end{equation}
defines the Ising-Vannimenus model with competing
nearest-neighbors and next-nearest-neighbors, where the sum in the
first term ranges all prolonged next-nearest-neighbors and the sum
in the second term ranges all nearest-neighbors and the spins
$\sigma(x)$ and $\sigma(y)$ take values in the set $\Phi$.  Here
$J_p,J\in \mathbb{R}$ are coupling constants corresponding to
prolonged next-nearest-neighbor and nearest-neighbor potentials,
respectively.

In \cite{Akin2016,NHSS}, the next generalizations are considered.
These authors have defined Gibbs measures or Gibbs states with
memory of length 2 (on spin-configurations $\sigma$) for
generalized ANNNI models on Cayley trees of order 2 with the
following formula:
\begin{equation}\label{mu}
\mu_{\textbf{h}}^{(n)}(\sigma)=\frac{1}{Z_{n}}\exp[-\beta
H_n(\sigma)+\sum_{x\in W_{n-1}}\sum_{y\in
S(x)}\sigma(x)\sigma(y)h_{xy,\sigma(x)\sigma(y)}].
\end{equation}
Here, as before, $\beta=\frac{1}{kT}$ and $\sigma_n: x\in V_n\to
\sigma_n(x)$ and $Z_n$ corresponds to the following partition
function:
$$
Z_n=\sum\limits_{\sigma_n\in \Omega_{V_n}}\exp[-\beta
H(\sigma_n)+\sum_{x\in W_{n-1}}\sum_{y\in
S(x)}\sigma(x)\sigma(y)h_{xy,\sigma(x)\sigma(y)}].
$$
Let us consider increasing subsets of the set of states for one
dimensional lattices \cite{FV} as follows:
$$\mathfrak{G}_1\subset \mathfrak{G}_2\subset... \subset \mathfrak{G}_n\subset...,
$$
where $\mathfrak{G}_n$ is the set of states corresponding to
non-trivial correlations between $n$-successive lattice points;
$\mathfrak{G}_1$ is the set of mean field states; and
$\mathfrak{G}_2$ is the set of Bethe-Peierls states, the latter
extending to the so-called Bethe lattices. All these states
correspond in probability theory to so-called Markov chains with
memory of length $n$ (see \cite{FV}).

In \cite{Akin2016,NHSS,NHSS1}, the authors have studied Gibbs
measures with memory of length 2 for generalized ANNNI models on a
Cayley tree of order 2 by means of a vector valued function
$$
\textbf{h}:<x,y> \rightarrow
\textbf{h}_{xy}=(h_{xy,++},h_{xy,+-},h_{xy,-+},h_{xy,--})\in
\mathbb{R}^4,
$$
where $h_{xy,\sigma(x)\sigma(y)}\in \mathbb{R}$ and $x\in W_{n-1},
y\in S(x).$


Let $x\in W_{n}$ for some $n$ and $S(x)=\{y,z,w\}$, where
$y,z,w\in W_{n+1}$ are the direct successors of $x$. Denote
$B_1(x)=\{x,y,z,w\}$ a unite semi-ball with a center $x$, where
$S(x)=\{y,z,w\}$.

We denote the set of all spin configurations on $V_n$ by
$\Phi^{V_n}$ and the set of all configurations on unite semi-ball
$B_1(x)$ by $\Phi^{B_1(x)}$. One can get that the set
$\Phi^{B_1(x)}$ consists of  sixteen configurations
\begin{equation}\label{config1}
\Phi^{B_1(x)}=\left\{\left(
  \begin{array}{ccc}
    l & k & j \\
      & i &
  \end{array}
\right): i,j,k,l\in \Phi \right\}.
\end{equation}
Let us denote the spin configurations belonging to $\Phi^{B_1(x)}$
by
\begin{equation*}
\sigma_1^{(1)}=\left(
  \begin{array}{ccc}
    + & + & + \\
      & + &
  \end{array}
\right), \sigma_2^{(1)}=\left(
  \begin{array}{ccc}
    + & + & - \\
      & + &
  \end{array}
\right),\sigma_3^{(1)}=\left(
  \begin{array}{ccc}
    + & - & + \\
      & + &
  \end{array}
\right), \sigma_4^{(1)}=\left(
  \begin{array}{ccc}
    - & + &+ \\
      & + &
  \end{array}
\right),
\end{equation*}
\begin{equation*}
\sigma_5^{(1)}=\left(
  \begin{array}{ccc}
    + & - & - \\
      & + &
  \end{array}
\right), \sigma_6^{(1)}=\left(
  \begin{array}{ccc}
    - & + & - \\
      & + &
  \end{array}
\right),\sigma_7^{(1)}=\left(
  \begin{array}{ccc}
    - & - & + \\
      & + &
  \end{array}
\right), \sigma_8^{(1)}=\left(
  \begin{array}{ccc}
    - & - &- \\
      & + &
  \end{array}
\right),
\end{equation*}

\begin{equation*}
\sigma_9^{(1)}=\left(
  \begin{array}{ccc}
    + & + & + \\
      & - &
  \end{array}
\right), \sigma_{10}^{(1)}=\left(
  \begin{array}{ccc}
    + & + & - \\
      & - &
  \end{array}
\right),\sigma_{11}^{(1)}=\left(
  \begin{array}{ccc}
    + & - & + \\
      & - &
  \end{array}
\right), \sigma_{12}^{(1)}=\left(
  \begin{array}{ccc}
    - & + &+ \\
      & - &
  \end{array}
\right),
\end{equation*}
\begin{equation*}
\sigma_{13}^{(1)}=\left(
  \begin{array}{ccc}
    + & - & - \\
      & - &
  \end{array}
\right), \sigma_{14}^{(1)}=\left(
  \begin{array}{ccc}
    - & + & - \\
      & - &
  \end{array}
\right),\sigma_{15}^{(1)}=\left(
  \begin{array}{ccc}
    - & - & + \\
      & - &
  \end{array}
\right), \sigma_{16}^{(1)}=\left(
  \begin{array}{ccc}
    - & - &- \\
      & - &
  \end{array}
\right).
\end{equation*}
For brevity, we adopt a natural definition for the quantities
$h\left(\begin{array} {ccc} z, y,w \\ x
 \end{array} \right)$ as $h_{B_1(x)}$.\\
By contrast, this paper assumes that vector valued function
$\textbf{h}:V\rightarrow \mathbb{R}^{16}$ is defined by
\begin{equation}\label{consistency}
\textbf{h}:<x,y,z,w>\rightarrow
\textbf{h}_{B_1(x)}=(h_{B_1(x),\sigma(x)\sigma(y)\sigma(z)\sigma(w)}:\sigma(x),\sigma(y),\sigma(z),\sigma(w)\in
\Phi),
\end{equation}
where $h_{B_1(x),\sigma(x)\sigma(y)\sigma(z)\sigma(w)}\in
\mathbb{R}$, $x\in W_{n-1}$ and $y,z,w\in S(x).$ Finally, we use
the function $h_{xyzw,\sigma(x)\sigma(y)\sigma(z)\sigma(w)}$ to
describe the Gibbs measure of any configuration $\left(
\begin{array} {ccc}\sigma(z)\ \ \sigma(y)\ \ \sigma(w)\\ \sigma(x)
\end{array}\right)$
that belongs to $\Phi^{B_1(x)}$.

\section{Construction of Gibbs measures and Functional Equations}\label{Gibbs measures}
On non-amenable graphs, Gibbs measures depend on boundary
conditions \cite{Rozikov}. In this paper, we consider this
dependency for Cayley trees, the simplest of graphs.
In this section, we present the general structure of Gibbs measures with memory of length 2 on the Cayley tree of order three.\\
An arbitrary edge $<x^{(0)},x^{(1)}>=\ell \in L$ deleted from a
Cayley tree $\Gamma^3_1$ and $\Gamma^3_0$ splits into two
components: semi-infinite Cayley tree $\Gamma^3_1$ and
semi-infinite Cayley tree $\Gamma^3_0$. This paper considers a
semi-infinite Cayley tree $\Gamma^3_0$ (see Fig. \ref{fig1}).
\begin{figure} [!htbp]\label{fig1}
\centering
\includegraphics[width=55mm]{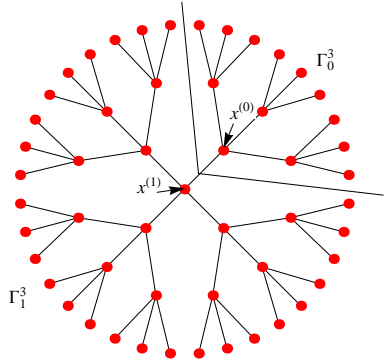}\ \ \ \ \ \ \ \ \ \ \ \
\caption{Cayley tree of order three, $k=3$.}\label{fig1}
\end{figure}
For a finite subset $V_n$ of the lattice, we define the
finite-dimensional Gibbs probability distributions on the
configuration space
$$\Omega^{V_n}=\{\sigma_n=\{\sigma(x)=\pm 1, x\in V_n \}\}
$$
at inverse temperature $\beta=\frac{1}{kT}$ by formula
\begin{eqnarray}\label{mu}
&&\mu_{\textbf{h}}^{(n)}(\sigma)\\\nonumber
&=&\frac{1}{Z_{n}}\exp[-\beta H_n(\sigma)+\sum_{x\in
W_{n-1}}\sum_{y,z,w\in
S(x)}\sigma(x)\sigma(y)\sigma(z)\sigma(w)h_{B_1(x),\sigma(x)\sigma(y)\sigma(z)\sigma(w)}].
\end{eqnarray}
with the corresponding partition function defined by
$$
Z_n=\sum\limits_{\sigma_n\in \Omega^{V_n}}\exp[-\beta
H(\sigma_n)+\sum_{x\in W_{n-1}}\sum_{y,z,w\in
S(x)}\sigma(x)\sigma(y)\sigma(z)\sigma(w)h_{B_1(x),\sigma(x)\sigma(y)\sigma(z)\sigma(w)}].
$$
We will obtain a new set of Gibbs measures that differ from
previous studies \cite{Akin2016,NHSS}. These new measures consider
translation-invariant boundary conditions. We will consider a
construction of an infinite volume distribution with given
finite-dimensional distributions. More exactly, we will attempt to
find a probability measure $\mu$ on $\Omega$  that is compatible
with given measures $\mu_{\textbf{h}}^{(n)}$, \emph{i.e.},
\begin{equation}\label{CM}
\mu(\sigma\in\Omega:
\sigma|_{V_n}=\sigma_n)=\mu^{(n)}_{\textbf{h}}(\sigma_n), \ \ \
\textrm{for all} \ \ \sigma_n\in\Omega^{V_n}, \ n\in \mathbf{N}.
\end{equation}

The consistency condition for $\mu_{\textbf{h}}^{n}(\sigma_n)$,
$n\geq 1$ is
\begin{equation}\label{comp}
\sum_{\omega\in\Omega^{W_n}}\mu^{(n)}_{\textbf{h}}(\sigma_{n-1}\vee\omega)=\mu^{(n-1)}_{\textbf{h}}(\sigma_{n-1}),
\end{equation}
for any configuration $\sigma_{n-1}\in\Omega^{V_{n-1}}$. This
condition implies the existence of a unique measure
$\mu_{\textbf{h}}$ defined on $\Omega$ with a required condition
\eqref{CM}. Such a measure $\mu_{\textbf{h}}$ is a Gibbs measure
with memory of length 2 corresponding to the model.

We define interaction energy on $V$ with the inner configuration
$\sigma_{n-1}\in\Omega^{V_{n-1}}$ and the boundary condition $\eta
\in \Omega^{ W_{n}}$ as
\begin{eqnarray}\label{ham1}
H_n(\sigma_{n-1}\vee\eta)
&=&-J\sum\limits_{<x,y>\in V_{n-1}}\sigma(x)\sigma(y) -J
\sum\limits_{x\in W_{n-1}}\sum\limits_{y\in
S(x)}\sigma(x)\eta(y)\\\nonumber &&-J_p\sum\limits_{>x,y<\in
V_{n-1}}\sigma(x)\sigma(y) -J_p \sum\limits_{x\in
W_{n-2}}\sum\limits_{z\in S^2(x)}\sigma(x)\eta(z)\\\nonumber
&=&H_n(\sigma_{n-1})-J\sum\limits_{x\in W_{n-1}}\sum\limits_{y\in
S(x)}\sigma(x)\eta(y)-J_p \sum\limits_{x\in
W_{n-2}}\sum\limits_{z\in S^2(x)}\sigma(x)\eta(z),
\end{eqnarray}
where $\sigma_{n-1}\vee\eta$ is the concatenation of the
configurations $\sigma_{n-1}$ and $\eta$. Thus, we have
\begin{eqnarray*}
&&\exp [-\beta H_n(\sigma_{n-1})+\sum\limits_{x\in W_{n-2}}
\sum\limits_{y,z,w\in S(x)}\sigma(x)\sigma(y)\sigma(z)\sigma(w)h_{B_1(x),\sigma(x)\sigma(y)\sigma(z)\sigma(w)}]\\
&=&L_n\sum\limits_{\eta\in \Omega^{W_{n}}}\exp[-\beta
H_n(\sigma_{n-1}\vee \eta)+\sum\limits_{y,z,w\in
W_{n-1}}\sum\limits_{y_i\in S(y)}\sum\limits_{z_i\in
S(z)}\sum\limits_{w_i\in S(w)}(B(h,J,J_p)],
\end{eqnarray*}
where $L_n=\frac{Z_{n-1}}{Z_n}$ and
\begin{eqnarray*}
B(h,J,J_p):&=&\sigma(y)\eta(y_1)\eta(y_2)\eta(y_3)h_{yy_1y_2y_3,\sigma(y)\eta(y_1)\eta(y_2)\eta(y_2)}\\
&+&\sigma(z)\eta(z_1)\eta(z_2)\eta(z_3) h_{zz_1z_2z_3,\sigma(z)\eta(z_1)\eta(z_2)\eta(z_3)})\\
&+&\sigma(w)\eta(w_1)\eta(w_2)\eta(w_3)
h_{ww_1w_2w_3,\sigma(w)\eta(w_1)\eta(w_2)\eta(w_3)}).
\end{eqnarray*}
Furthermore, equation \eqref{ham1} provides that
\begin{eqnarray*}\label{Kolmogorov1}
&&\exp [-\beta H_n(\sigma_{n-1})+\sum\limits_{x\in
W_{n-2}}\sum\limits_{y,z,w\in S(x)}
\sigma(x)\sigma(y)\sigma(z)\sigma(w)h_{xyzw,\sigma(x)\sigma(y)\sigma(z)\sigma(w)}]\\
&=&L_n\sum\limits_{\eta\in \Omega^{W_{n}}}\exp[-\beta
H_n(\sigma_{n-1})-\beta J\sum\limits_{x\in
W_{n-1}}\sum\limits_{y\in S(x)}\sigma(x)\eta(y)\\
&&-\beta J_p\sum\limits_{x\in W_{n-2}}\sum\limits_{z\in
S^2(x)}\sigma(x)\eta(z)+\sum\limits_{y,z,w\in
W_{n-1}}\sum\limits_{y_i\in S(y)}\sum\limits_{z_i\in
S(z)}\sum\limits_{w_i\in S(w)}B(h,J,J_p)].\\\nonumber
\end{eqnarray*}
For all $i=1,2,3$, we get
\begin{eqnarray}\label{Kolmogorov2}
&&\prod\limits_{x\in W_{n-2}}\prod\limits_{y,z,w\in
S(x)}e^{[\sigma(x)\sigma(y)\sigma(z)\sigma(w)h_{xyzw,\sigma(x)\sigma(y)\sigma(z)\sigma(w)}]}\\\nonumber
&=&L_n\prod\limits_{x\in W_{n-2}}\prod\limits_{y,z,w\in
S(x)}\prod\limits_{y_i\in S(y)}\prod\limits_{z_i\in
S(z)}\prod\limits_{w_i\in
S(w)}\sum\limits_{\eta(x_i),\eta(y_i),\eta(z_i),\eta(w_i)\in
\Phi}e^{[A(h,J,J_p)]},\\\nonumber
\end{eqnarray}
where
\begin{eqnarray*}
A(h,J,J_p)&=&\sigma(y)\eta(y_1)\eta(y_2)\eta(y_3)h_{yy_1y_2y_3,\sigma(y)\eta(y_1)\eta(y_2)\eta(y_3)}\\
&+&\sigma(z)\eta(z_1)\eta(z_2)\eta(z_3) h_{zz_1z_2z_3,\sigma(z)\eta(z_1)\eta(z_2)\eta(z_3)})\\
&+&\sigma(w)\eta(w_1)\eta(w_2)\eta(w_3) h_{ww_1w_2w_3,\sigma(w)\eta(w_1)\eta(w_2)\eta(w_3)})\\
&+&\beta[J(\sigma(y)(\eta(y_1)+\eta(y_2)+\eta(y_3))+\sigma(z)(\eta(z_1)+\eta(z_2)\\
&&+\eta(z_3)+\sigma(z)(\eta(z_1)+\eta(z_2)+\eta(z_3))]\\
&+&\beta\left[J_p\sigma(x)(\sum_{i=1}^3(\eta(w_i)+\eta(y_i)+\eta(z_i)))\right].
\end{eqnarray*}

\begin{figure} [!htbp]\label{fig2}
\centering
\includegraphics[width=55mm]{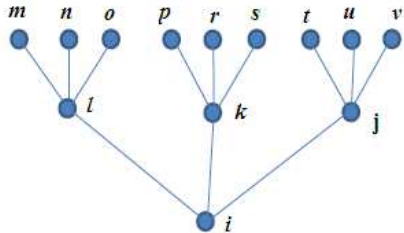}
\caption{Configurations on semi-finite Cayley tree of order three
with levels 2}\label{fig2}
\end{figure}
Next, let us fix $<x,y>$, $<x,z>$ and $<x,w>$ by rewriting
\eqref{Kolmogorov2} for all values of $\sigma(x), \sigma(y),
\sigma(z),\sigma(w)\in \Phi $. For the sake of simplicity, we
assume $\sigma(x)=i$, $\sigma(y)=j$, $\sigma(z)=k$, $\sigma(w)=l$,
$\eta(y_1)=u$, $\eta(y_2)=v$, $\eta(y_3)=t$, $\eta(z_1)=s$,
$\eta(z_2)=r$, $\eta(z_3)=p$, $\eta(w_1)=0$, $\eta(w_2)=n$,
$\eta(w_3)=m$,
where $i,j,k,l,m,n,o,p,r,s,t,u,v\in \Phi $ (see Figure \ref{fig2}).\\
Then from equation \eqref{Kolmogorov2}, we can obtain an explicit
expression
\begin{eqnarray}\label{exponent1}
e^{ijkl\mathbf{h}_{B_1(x),i;j,k,l}}&=&L_2\sum\limits_{m,n,o,p,r,s,t,u,v\in
\Phi}[e^{\beta J_p i(m+n+o+p+r+s+t+u+v)}\\\nonumber
&\times &e^{\beta J(l(m+n+o)+k(p+r+s)+j(t+u+v))}\\
&\times
&e^{jtuv\mathbf{h}_{B_1(y),j;t,u,v}+kprs\mathbf{h}_{B_1(z),k;s,r,p}+lmno\mathbf{h}_{B_1(w),l;o,n,m}}],\nonumber
\end{eqnarray}
where $L_2=\frac{Z_{1}}{Z_{2}}$.\\
Let
\begin{eqnarray}\label{h-funct}
&&h_1=h_{B_1(x),\sigma_1^{(1)}},\\
&&h_2=h_{B_1(x),\sigma_2^{(1)}}=h_{B_1(x),\sigma_3^{(1)}}=h_{B_1(x),\sigma_4^{(1)}},\\
&&h_3=h_{B_1(x),\sigma_5^{(1)}}=h_{B_1(x),\sigma_6^{(1)}}=h_{B_1(x),\sigma_7^{(1)}},\\
&&h_4=h_{B_1(x),\sigma_8^{(1)}},
\end{eqnarray}
\begin{eqnarray}\label{h-funct5}
&&h_5=h_{B_1(x),\sigma_9^{(1)}},\\
&&h_6=h_{B_1(x),\sigma_{10}^{(1)}}=h_{B_1(x),\sigma_{11}^{(1)}}=h_{B_1(x),\sigma_{12}^{(1)}},\\
&&h_7=h_{B_1(x),\sigma_{13}^{(1)}}=h_{B_1(x),\sigma_{14}^{(1)}}=h_{B_1(x),\sigma_{15}^{(1)}},\\\label{h-funct8}
&&h_8=h_{B_1(x),\sigma_{16}^{(1)}}.
\end{eqnarray}
Therefore, we can redefine the vector-valued function given in
\eqref{consistency}  as follows:
\begin{equation}\label{consistency1}
\textbf{h}(x)=(h_1,h_2,h_3,h_4,h_5,h_6,h_7,h_8).
\end{equation}

\subsection{Basic Equations}

Assume that $a=e^{\beta J}$ and $b=e^{\beta J_p}$. By using the
equations \eqref{h-funct}-\eqref{h-funct8}, we can take new
variables $u'_{i}=e^{h_{B_1(x),\sigma_{j}^{(1)}}}$ for $x\in
W_{n-1}$ and $u_{i}=e^{h_{B_1(y),\sigma_{j}^{(1)}}}$ for $y\in
S(x)$.
For convenience, we will use a shorter notation for the recurrence
system \cite{Vannimenus}.
From \eqref{exponent1}, through direct enumeration, we obtain the
following eight equations:
\begin{eqnarray}\label{recurrence1aa}
u'_{1}&=&L_2\left((ab)^3u_1+\frac{3ab}{u_2}+\frac{3u_3}{ab}+\frac{1}{(ab)^3u_4}\right)^3,\\
(u'_{2})^{-1}&=&L_2\left((ab)^3u_1+\frac{3ab}{u_2}+\frac{3u_3}{ab}+\frac{1}{(ab)^3u_2}\right)^2\nonumber\\
&\times&\left(\frac{b^3}{a^3u_5}+\frac{3bu_6}{a}+\frac{3a}{bu_7}+\frac{a^3u_8}{b^3}\right),\\\nonumber
u'_{3}&=&L_2\left((ab)^3u_1+\frac{3ab}{u_2}+\frac{3u_3}{ab}+\frac{1}{(ab)^3u_2}\right)\nonumber\\
&\times&\left(\frac{b^3}{a^3u_5}+\frac{3bu_6}{a}+\frac{3a}{bu_7}+\frac{a^3u_8}{b^3}\right)^2,\\
(u'_{4})^{-1}&=&L_2\left(\frac{b^3}{a^3u_5}+\frac{3bu_6}{a}+\frac{3a}{bu_7}+\frac{a^3u_8}{b^3}\right)^3,
\end{eqnarray}

\begin{eqnarray}\label{recurrence2aa}
(u'_{5})^{-1}&=&L_2\left(\frac{a^3u_1}{b^3}+\frac{3a}{bu_2}+\frac{3bu_3}{a}+\frac{b^3}{a^3u_4}\right)^3,\\
u'_{6}&=&L_2\left(\frac{a^3u_1}{b^3}+\frac{3a}{bu_2}+\frac{3bu_3}{a}+\frac{b^3}{a^3u_4}\right)^2\nonumber\\
&\times&\left(\frac{1}{(ab)^3u_5}+\frac{3u_6}{ab}+\frac{3ab}{u_7}+(ab)^3u_8\right),\\
(u'_{7})^{-1}&=&L_2\left(\frac{a^3u_1}{b^3}+\frac{3a}{bu_2}+\frac{3bu_3}{a}+\frac{b^3}{a^3u_4}\right)\nonumber\\
&\times&\left(\frac{1}{(ab)^3u_5}+\frac{3u_6}{ab}+\frac{3ab}{u_7}+(ab)^3u_8\right)^2,\\
u'_{8}&=&L_2\left(\frac{1}{(ab)^3u_5}+\frac{3u_6}{ab}+\frac{3ab}{u_7}+(ab)^3u_8\right)^3.\label{recurrence2aa4}
\end{eqnarray}

From the equations \eqref{recurrence1aa}-\eqref{recurrence2aa4},
it is obvious that
\begin{eqnarray*}
(u'_2)^{3}&=&\frac{(u'_4)}{(u'_1)^{2}},\\
 (u'_3)^{3}&=&\frac{(u'_1)}{(u'_4)^{2}},\\
(u'_6)^{3}&=&\frac{(u'_8)}{(u'_5)^{2}},\\
(u'_7)^{3}&=&\frac{(u'_5)}{(u'_8)^{2}}.
\end{eqnarray*}
Therefore, selecting variables $u'_1$, $u'_4$ $u'_5$ and $u'_8$,
we obtain only 4 variables.
\begin{rem} If the vector-valued function $\textbf{h}(x)$ given in \eqref{consistency1} has the following form:
\begin{equation*}
\textbf{h}(x)=(p,\frac{q-2p}{3},\frac{p-2q}{3},q,r,\frac{s-2r}{3},\frac{r-2s}{3},s),
\end{equation*}
then the \emph{consistency} condition \eqref{comp} is satisfied,
where $p,q,r,s\in \mathbb{R}$.
\end{rem}
Considering new variables $u_i=v_i^{3}$ for $i=1,4,5,8$, following
recurrent equations a new recurrence system can be expressed in a
simpler form:
\begin{eqnarray}\label{recurrence4}
(v'_{1})&=&\sqrt[3]{L_2}\left(\frac{1+(ab)^{2}v_{1}v_{4}}{abv_{4}}\right)^3,\\\label{recurrence4a}
(v'_{4})^{-1}&=&\sqrt[3]{L_2}\left(\frac{b^{2}+a^{2}v_{5}v_{8}}{abv_{5}}\right)^3,\\\label{recurrence4b}
(v'_{5})^{-1}&=&\sqrt[3]{L_2}\left(\frac{b^{2}+a^{2}v_{1}v_{4}}{abv_{4}}\right)^3,\\\label{recurrence4c}
(v'_{8})&=&\sqrt[3]{L_2}\left(\frac{1+(ab)^{2}v_{5}v_{8}}{abv_{5}}\right)^3.\label{recurrence4d}
\end{eqnarray}
The solutions of this system of nonlinear equations
\eqref{recurrence4}-\eqref{recurrence4c} describe
translationЦinvariant Gibbs measures.
\section{Translation-invariant Gibbs measures}\label{TIGM}
In this section, we are going to focus on the existence of
translation-invariant Gibbs measures (TIGMs) by analyzing the
equation \eqref{exponent1}. Note that a function $\textbf{h} =
\{h_{B_1(x),\sigma_i^{(1)}}:i\in \{1,2,\ldots,16\} \}$ is
considered as translation-invariant if $h_{B_1(x),\sigma_i^{(1)}}=
h_{B_1(y),\sigma_i^{(1)}}$ for all $y\in S(x)$ and $i\in
\{1,2,\ldots,16\}$. A translation-invariant Gibbs measure is
defined as a measure, $\mu_{\textbf{h}}$, corresponding to a
translation-invariant function $\textbf{h}$ (see for details
\cite{NHSS,Rozikov}). Here we will assume that $v'_{i}=v_{i}$ for
all $i\in \{1,4,5,8\}$.\\
The analysis of the solutions of the system of equations
\eqref{recurrence4}-\eqref{recurrence4d} is rather tricky. Below
we will consider the following case when the system of equations
\eqref{recurrence4}-\eqref{recurrence4d} is solvable for set

\begin{eqnarray}\label{invariantA}
A=\left\{(v_1,v_4,v_5,v_8)\in
\mathbb{R}_+^4:v_{1}=v_{4}^{3},v_{8}=v_{5}^{3}\right\}.
\end{eqnarray}
Now, we want to find Gibbs measures for  considered case. To do
so, we introduce some notations. Define the transformation

\begin{eqnarray}\label{cavity}
\textbf{F}=(F_1, F_4, F_5, F_8): \mathbf{R}^4_+ \rightarrow
\mathbf{R}^4_+
\end{eqnarray}
with $v'_1=F_1(v_1,v_4,v_5,v_8)$, $v'_4=F_4(v_1,v_4,v_5,v_8)$,
$v'_5=F_5(v_1,v_4,v_5,v_8)$ and $v'_8=F_8(v_1,v_4,v_5,v_8)$. The
fixed points of the cavity equation
$\textbf{v}=\textbf{F}(\textbf{v})$ given in the Eq.
\eqref{cavity} describe the translation-invariant Gibbs measures
of the Ising model corresponding to the Hamiltonian \eqref{hm},
where $\textbf{v}=(v_1,v_4,v_5,v_8)$.

Divide \eqref{recurrence4} by \eqref{recurrence4a}, then we have
\begin{eqnarray}\label{recurrence4e}
v_4^7v_5^{-3}=\left(\frac{1+(ab)^2v_4^4}{b^2+a^2v_5^4}\right)^3.
\end{eqnarray}
Similarly, divide \eqref{recurrence4c} by \eqref{recurrence4b},
then one gets
\begin{eqnarray}\label{recurrence4f}
v_5^7v_4^{-3}=\left(\frac{1+(ab)^2v_5^4}{b^2+a^2v_4^4}\right)^3.
\end{eqnarray}
%
Multiply the equations \eqref{recurrence4e} and
\eqref{recurrence4f}, we obtain
\begin{eqnarray*}\label{recurrence5}
v_4^4v_5^{4}=\left(\frac{1+(ab)^2v_4^4}{b^2+a^2v_5^4}\right)^3\left(\frac{1+(ab)^2v_5^4}{b^2+a^2v_4^4}\right)^3.
\end{eqnarray*}
Assume that $v_4^4=v_5^{4}=x$, then we get
\begin{eqnarray}\label{recurrence5a}
x=\left(\frac{1+(ab)^2x}{b^2+a^2x}\right)^3.
\end{eqnarray}
Therefore, we will study the following nonlinear dynamical system
\begin{equation*}\label{function7a}
f(x)=\left(\frac{1+(ab)^2x}{b^2+a^2x}\right)^3.
\end{equation*}
Let us find the fixed points of the function $f.$ For brevity, we
assume that $e^{2J/T}=a^{2}=c$ and $ e^{2J_p/T}=b^{2}=d,$ where
$T$ is the absolute temperature.

One can show that the function $f$ is conjugate to the following
function
\begin{equation}\label{function7}
g(x)=\left(\frac{1+cd x}{d+c x}\right)^3.
\end{equation}
Thus, the study of the problem of phase transition for the
considered model \eqref{hm} is reduced to the investigation of the
fixed points of nonlinear dynamical system \eqref{function7}.

Creating conditions favorable to the occurrence of phase
transition depends in part on finding a so-called critical
temperature. Note that the equations above describe the fixed
points of equation \eqref{function7}, which satisfies the
consistency condition. When there is more than one solution for
the equation \eqref{function7}, then more than one
translation-invariant Gibbs measure corresponds to those
solutions. Thus, the equation \eqref{function7} have more than one
positive solution, a phase transition occurs for model \eqref{hm}.
This possible non-uniqueness corresponds in the language of
statistical mechanics to the phenomenon of phase transition
\cite{Preston}. Phase transitions usually occur at low
temperatures. Finding an exact value for $T_c$, where $T_c$ is the
critical value of temperature, can enable the creation of
conditions in which a phase transition occurs for all $T$. Solving
models for $T_c$ will lead to finding the exact value of the
critical temperatures.

The number of fixed points of the function \eqref{function7}
naturally depends on the parameters $\beta=1/kT$ and the coupling
constants $J$ and $J_p$. Thus, we will find positive fixed points
of the nonlinear dynamical system \eqref{function7}. Therefore,
the fixed points of the cavity equation
$\textbf{v}=\textbf{F}(\textbf{v})$ given in the Eq.
\eqref{cavity} will describe translation-invariant Gibbs measures
with memory of length 2 for the Ising-Vannimenus model under
conditions \eqref{invariantA}, where
$\textbf{v}= (v_1,v_4, v_5,v_8)$.\\
Let us now investigate the fixed points of the dynamic system
\eqref{function7}, i.e., $x = g(x)$.
If we define $g: \mathbb{R}^{+}\rightarrow \mathbb{R}^{+}$ then
$g$ is bounded and thus the curve $y = g(x)$ must intersect the
line $y =m x.$ Therefore, this construction provides one element
of a new set of Gibbs measures with memory of length 2,
corresponding to the model  \eqref{hm} for any $x\in
\mathbb{R}^{+}$.
\begin{prop}\label{proposition}
The equation
\begin{eqnarray}\label{recurrence16b}
x=\left(\frac{1+cd x}{d+c x}\right)^3
\end{eqnarray}
(with $x \geq 0, c > 0, d > 0$) has one solution if  $d<1$. If $d
> 2$ then there exists $\eta_1(c,b)$, $\eta_2(c,d)$ with
$0<\eta_1(c,d)<\eta_2(c,d)$ such that equation
\eqref{recurrence16b} has 3  solutions if
$\eta_1(c,d)<m<\eta_2(c,d)$ and has 2 solutions if either
$\eta_1(c,d)=m$ or $\eta_2(c,d)=m$, where
\begin{eqnarray*}
\eta_1 (c,d)&=&-\frac{c d^4 \left(1-d^2+\sqrt{4-5 d^2+d^4}\right)^3}{\left(2-2 d^2+\sqrt{4-5 d^2+d^4}\right)^3 \left(2-d^2+\sqrt{4-5 d^2+d^4}\right)}\\
\eta_2 (c,d)&=&\frac{c d^4 \left(-1+d^2+\sqrt{4-5
d^2+d^4}\right)^3}{\left(-2+d^2+\sqrt{4-5 d^2+d^4}\right)
\left(-2+2 d^2+\sqrt{4-5 d^2+d^4}\right)^3}.
\end{eqnarray*}

\end{prop}
\begin{proof} Let
$$
g(x)=\left(\frac{1+c d x}{d+c x}\right)^3.
$$
Then, taking the first and the second derivatives of the function
$g$, we have
$$
g'(x)=\frac{3 c (d^2-1)(1+c d x)^2}{(d+c x)^4},
$$
$$
g''(x)=-\frac{6 c^2(d^2-1)(1+c d x) \left(2-d^2+c d x\right)}{(d+c
x)^5}.
$$
If $d<1$ (with $x \geq0$) then $g$ is decreasing and there can
only be one solution of $g(x)=x.$ Thus, we can restrict ourselves
to the case in which $d>1.$  It is not hard to show by simple
calculus arguments that the graph of $y=g(x)$ over interval
$(0,\frac{d^2-2}{cd})$ is concave up  and the graph of $y=g(x)$
over interval $(\frac{d^2-2}{cd},\infty)$ is concave down. As a
result, there are at most 3 positive solutions for $g(x) =x.$
According to Preston \cite[Proposition 10.7]{Preston}, there can
be more than one solution if and only if there is more than one
solution to $xg'(x) = g(x)$, which is the same as
\begin{equation}\label{root1}
c^2 d x^2-2c\left( d^2-2\right) x+d=0.
\end{equation}
With some elementary analysis, it is clear that if $d>\sqrt{2}$
and $(d^2-4) (d^2-1)>0$ then the quadratic equation \eqref{root1}
has 2 solutions. The solutions are
$$
x^{*}_1=\frac{(d^2-2)-\sqrt{ \left(4-5 d^2+d^4\right)}}{c d},\ \ \
\ \ \ \ x^{*}_2=\frac{(d^2-2)+\sqrt{ \left(4-5 d^2+d^4\right)}}{c
d},
$$
where $d>2$ due to $d >1$. Then $g'(x^{*}_1) < 1$ and $g'(x^{*}_2)
> 1$. That is, $g(x^{*}_1) <x^{*}_1$ and $g(x^{*}_2) > x^{*}_2$,
if $\eta_1(c,d) < 1 <\eta_2(c,d)$. So, the proof is readily
completed.
\end{proof}
If the collection ${h_{B_1(x)}, x\in V_{0}}$ satisfies the
equation \eqref{exponent1} for any $x\in V_{0}$, then
$|h_{B_1(x)}|\leq h_{*}$, for any $x\in V^{0}$, and if
$h_{B_1(x)}= h_{*}$, (or $h_{B_1(x)}=-h_{*}$), then
$h_{B_1(y)}=h_{*}$, (respectively, $h_{B_1(y)}=-h_{*}$,) for any
$y\geq x$ (see for details \cite{BG}).

It is very important that the equation \eqref{exponent1} describes
all the relations between the quantities $\{h_{B_1(x)}, x\in
V_{0}\}$.

\subsection{An illustrative example}
\begin{figure} [!htbp]\label{fig1Cregion}
\centering
\includegraphics[width=60mm]{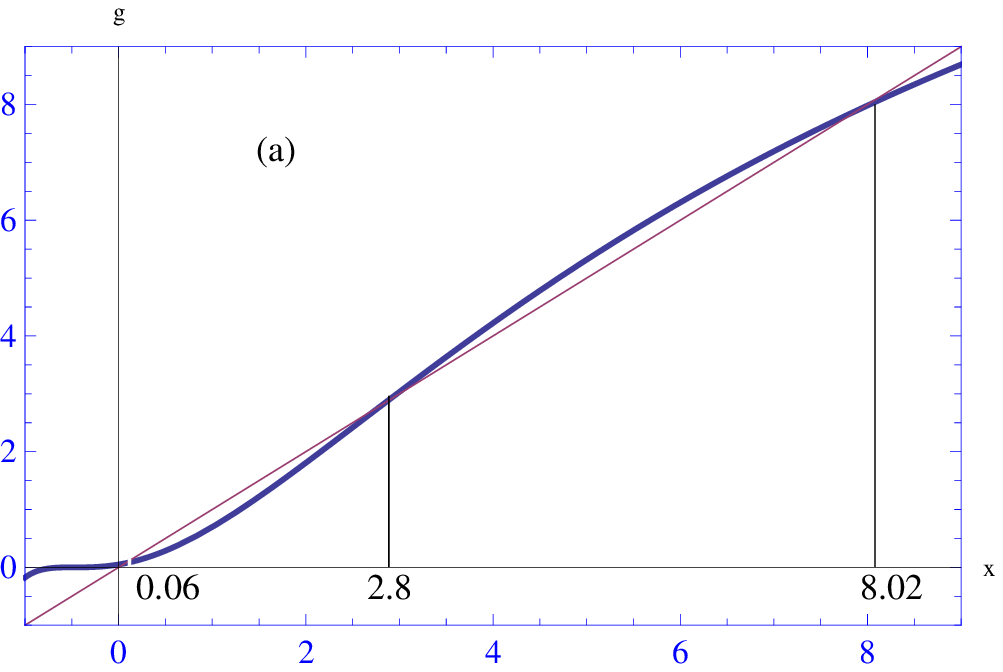}\label{fig1c}
\includegraphics[width=60mm]{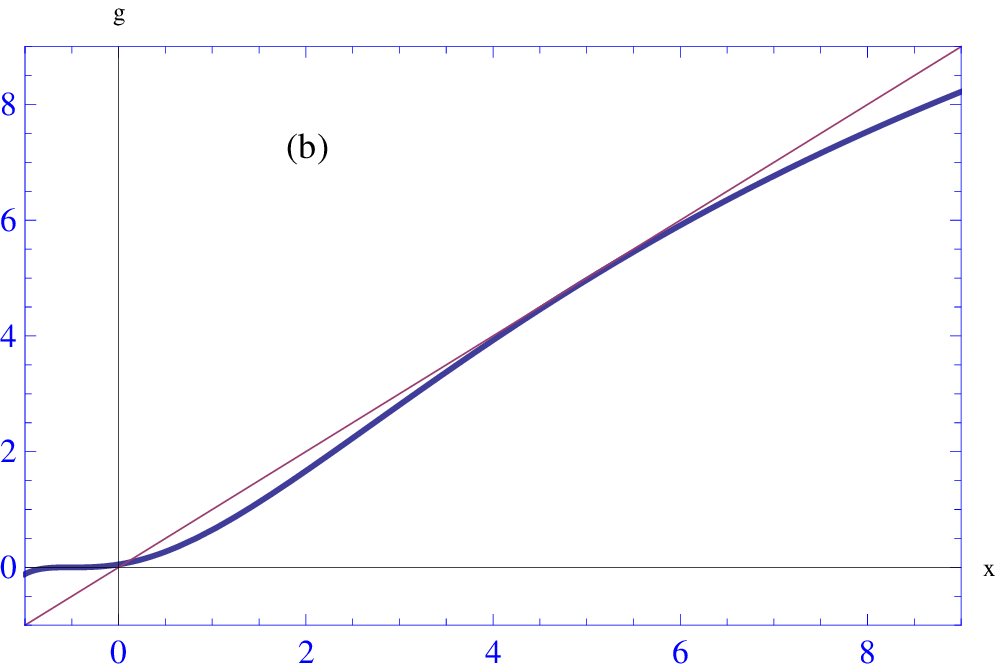} \label{fig1c}
\includegraphics[width=60mm]{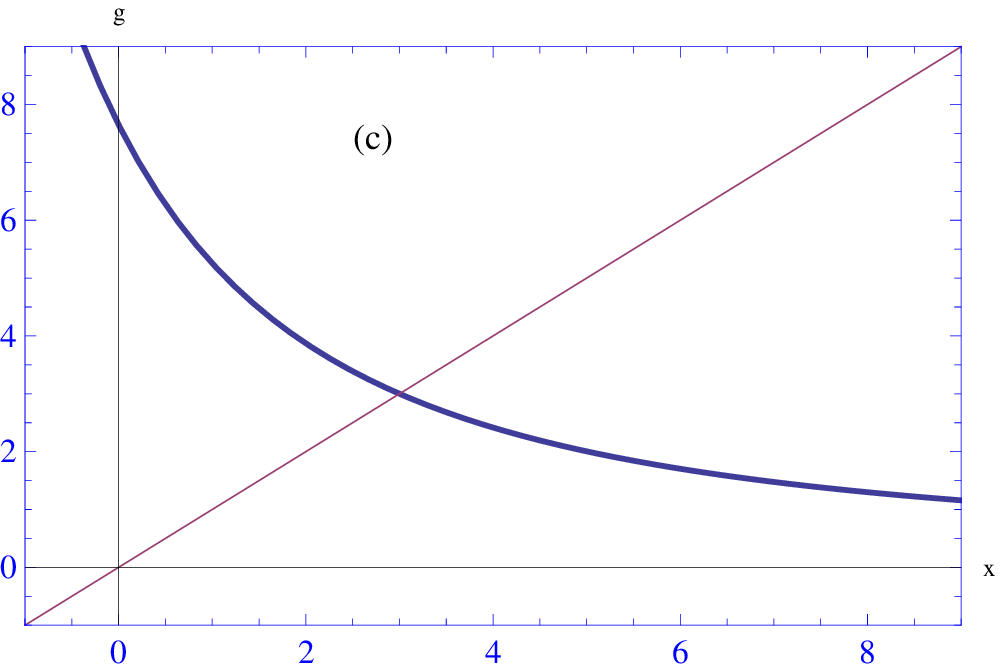}
\caption{(a) There exist three positive roots of the equation
\eqref{function7} for $J = -1.7, J_p =6.5, T = 13$. (b) There
exist two positive roots of the equation \eqref{function7} for $J
= -1.045, J_p = -1.045, T = 6.55.$ (c) There exists only one
positive root of the equation \eqref{function7} for $J = 6.75, J_p
= 1.95, T = -5.75$.}\label{fig1Cregion}
\end{figure}
Elementary analysis allows us to obtain the fixed points of the
function \eqref{function7} by finding real positive roots of
equation \eqref{recurrence16b}. Thus, we can obtain a polynomial
equation of degree 4. Previously documented analysis has solved
these equations, which we will not show here due to the
complicated nature of formulas and coefficients \cite{Wolfram}.
Nonetheless, we have manipulated the polynomial equation via
Mathematica \cite{Wolfram}.
We have obtained 3 positive real roots for some parameters $J$ and
$J_p$ (coupling constants) and temperature $T$. As an illustrative
example,  Fig. \ref{fig1Cregion} (a) shows that there are  3
positive fixed points of the function \eqref{function7} for $J =
-1.7, J_p =6.5, T = 13$ and $m=1$. Therefore, for $J = -1.7, J_p
=6.5, T = 13$ we have demonstrated the occurrence of phase
transitions. Fig. \ref{fig1Cregion} (b) shows that there are two
positive fixed points of the function $g$ for $J = -1.045, J_p =
-1.045, T = 6.55$. In Figure \ref{fig1Cregion} (c), there exists
only single positive fixed point of the function \eqref{function7}
for $J = 6.75, J_p = 1.95, T = -5.75$. Therefore,  the phase
transition does not occur for $J = 6.75, J_p = 1.95, T = -5.75$.

From the Fig. \ref{fig1Cregion} (a), let us consider
$x_1^{*}\approx 0.06,x_2^{*}\approx 2.8,x_3^{*}\approx 8.02$.
Figure \ref{fig1Cregion}a illustrates that for all $x\in
(x_2^{*},x_3^{*})$, $\lim\limits_{n\rightarrow
\infty}g^n(x)=x_3^{*}.$ Similarly, for all $x\in
(x_1^{*},x_2^{*})$, $\lim\limits_{n\rightarrow
\infty}g^n(x)=x_1^{*}.$ Therefore, the fixed points $x_1^{*}$ and
$x_3^{*}$ are stable and $x_2^{*}$ is unstable.

Therefore, there is a critical temperature $T_c > 0$ such that for
$T < T_c$ this system of equations has 3 positive solutions:
$h_1^{*}; h_2^{*};h_3^{*}.$  We denote the Gibbs measure that
corresponds to the root $h_1^{*}$ (and respectively
$h_2^{*};h_3^{*}$) by $\mu^{(1)}$ (and respectively
$\mu^{(2)}$,$\mu^{(3)}$).

\begin{rem}
Note that the stable roots describe extreme Gibbs distributions.
Therefore, from the Figure \ref{fig1Cregion} (a), we can conclude
that the Gibbs measures  $\mu_1^{*}$ and $\mu_3^{*}$ corresponding
to the stable fixed points $x_1^{*}$ and $x_3^{*}$ are extreme
Gibbs distributions (see for details
\cite{Iof,NHSS,NHSS1,Rozikov}).
\end{rem}
\begin{rem}
We conclude that there are at most 3 translation-invariant Gibbs
measures corresponding to the positive real roots of the equation
\eqref{recurrence16b}. Also, one can show that
translation-invariant Gibbs measures corresponding stable
solutions are extreme.
\end{rem}

\section{Conclusions}\label{Conclusions}
In this paper, by using a new approach to obtain Gibbs measures of
Vannimenus model on a Cayley tree of order three, we have
constructed the recurrence equations corresponding to the model.
The Kolmogorov \emph{consistency} condition has been satisfied. We
have investigated the translation-invariant Gibbs measures
associated to the set $A$ given in \eqref{invariantA}. By means of
such constructions, we have studied the existence of phase
transition for translation-invariant Gibbs measures. The complete
characterization of the extremal measures at any inverse
temperature $\beta=\frac{1}{T}$ remains an important issue.

We stress that the specified model was investigated only
numerically, without rigorous mathematical proofs \cite{NHSS}.
This paper has thus proposed a rigorous measure-theoretical
approach to investigate Gibbs measures with memory of length 2
corresponding to the Ising-Vannimenus model on a Cayley tree of
order three. In this paper, we have also obtained new Gibbs
measures different from the Gibbs measures given in the references
\cite{Akin2016,NHSS}.

Note that in \cite{Akin2016} we established the existence,
uniqueness or non-uniqueness of the translation-invariant Gibbs
measures associated with the Ising-Vannimenus model corresponding
to the same Hamiltonian \eqref{hm} on the Cayley tree of order
two. Hence, results of the present paper totaly differ from
\cite{Akin2016}, and show by increasing the dimension of the tree
we are getting the phase transition for some given parameters $J,
J_p, T$. For example, one can easily examine that although the
phase transition of the same model does not occur for $J = -1.7,
J_p = 6.5, T = 13, k=2$ (see the First Case in \cite{Akin2016}),
there is the phase transition of the Ising-Vannimenus model
corresponding to the Hamiltonian \eqref{hm} for the parameters $J
= -1.7, J_p = 6.5, T = 13, k=3$. Also, depending on the even and
odd of $k$, the recurrence equations obtained in the present paper
totaly differ from \cite{Akin2016}. Exact description of the
solutions of the system of equations
\eqref{recurrence4}-\eqref{recurrence4c} is rather tricky.
Therefore, we have considered only one case \eqref{invariantA},
the other cases remain open problem.

Note that the grid ${{\mathbb{Z}}^{d}}$ is the Cayley tree of the
free abelian group with $d$ generators. $d$-dimensional integer
lattice, denoted ${{\mathbb{Z}}^{d}}$, has so-called amenability
property \cite{Preston}. Moreover, analytical solutions do not
exist on such lattice. But investigations of phase transitions of
spin models on hierarchical lattices show that there are exact
calculations of various physical quantities \cite{Peruggi}. For
many problems the solution on a tree is much simpler than on a
regular lattice such as $d$-dimensional integer lattice and is
equivalent to the standard Bethe-Peierls theory \cite{Katsura}.
The Cayley tree is not a realistic lattice; however, its amazing
topology makes the exact calculations of various quantities
possible. Therefore, the results obtained in our paper can inspire
to study Ising and Potts models over multi-dimensional lattices or
the grid ${{\mathbb{Z}}^{d}}$.

To our knowledge this is the first example of the rigorous study
of Gibbsian phenomena related to the Markov random field on the
Cayley tree of order three by using our approach. The
investigations of the problem for arbitrary order ($k>3$) are very
difficult. Therefore, we will study new results related to the
model for arbitrary order ($k>3$) in a next paper.

We believe the method used here can be applied to any other
lattice model studied in the literature \cite{UA1,UA2}. By
considering the method used in this paper, one can study new Gibbs
measures with memory of length $n>2$ associated with Ising model
on arbitrary odd order Cayley tree and Cayley tree-like lattices
\cite{Moraal}.

\textbf{Acknowledgments}\\
We wish to thank the anonymous referees for useful suggestions
which substantially improved this paper.

\end{document}